\documentclass[11pt]{article}
\usepackage{amscd}
\usepackage{amsfonts}
\usepackage{amsmath}
\usepackage{amssymb}
\usepackage{amsthm}
\usepackage{bbm}
\usepackage{CJK}
\usepackage{fancyhdr}
\usepackage{graphicx}
\usepackage{hyperref}
\usepackage{indentfirst}
\usepackage{latexsym,bm}
\usepackage{mathrsfs}
\usepackage{xypic}

\newtheorem{theorem}{Theorem}[section]
\newtheorem{lemma}[theorem]{Lemma}
\newtheorem{definition}[theorem]{Definition}

\newtheorem{remark}[theorem]{Remark}

\usepackage[top=1in,bottom=1in,left=1.25in,right=1.25in]{geometry}
\def\<{\langle}
\def\>{\rangle}
\def\a{\alpha}
\def\b{\beta}

\def\g{\gamma}

\def\l{\lambda}

\def\o{\otimes}

\date{}
\begin{document}
\renewcommand{\baselinestretch}{1.2}
\renewcommand{\arraystretch}{1.0}
\title{\bf Yetter-Drinfeld category for the quasi-Turaev group coalgebra}
 \date{}
\author {{\bf Daowei Lu\footnote {Corresponding author:  Daowei Lu, ludaowei620@126.com,
 Department of Mathematics, Southeast University, Nanjing, Jiangsu 210096, P. R. of China.}, \quad  Shuanhong Wang }\\
{\small Department of Mathematics, Southeast University}\\
{\small Nanjing, Jiangsu 210096, P. R. of China}}
 \maketitle
\begin{center}
\begin{minipage}{12.cm}

\noindent{\bf Abstract.} Let $\pi$ be a group. The aim of this paper is to construct the category of Yetter-Drinfeld modules over the quasi-Turaev group coalgebra $H=(\{H_\a\}_{\a\in\pi},\Delta,\varepsilon,S,\Phi)$, and prove that this category is isomorphic to the center of the representation category of $H$. Therefore a new Turaev braided group category is constructed.
\\

\\

 \end{minipage}
 \end{center}
 \normalsize\vskip1cm

\section*{Introduction}
\setcounter{equation} {0} \hskip\parindent

Given a group $\pi$, Turaev in \cite{T2} introduced the notion of a braided $\pi$-monoidal
category which is called $Turaev~braided~group~category$ in this paper, and showed that such a category
gives rise to a 3-dimensional homotopy quantum field theory. Meanwhile such a category plays a key role in the construction of Hennings-type invariants of flat group-bundles over complements of link in the 3-sphere, see \cite{V2}.

For the above reasons, it becomes very important to construct Turaev braided group category. Based on the work of \cite{PS}, more results have been obtained in \cite{LW} and \cite{YW}, where the method used in \cite{PS} was applied to weak Hopf algebras and regular multiplier Hopf algebras. It is well-known that there is anther approach to the construction, for instance, in \cite{FW} the authors introduced the notion of quasi-Hopf group coalgebras and proved that the representation category of quasitriangular quasi-Hopf group coalgebras is exactly a Turaev braided group category.

M. Zunino in \cite{Z} constructed the Yetter-Drinfeld category of crossed Hopf group coalgebra and showed that it is a Turaev braided group category. Motivated by this construction, in this paper, we will generalize this result to quasi-Turaev Hopf group coalgebra defined in \cite{FW}. The notion of Yetter-Drinfeld category of quasi-Turaev Hopf group coalgebra will be given, and the isomorphism between Yetter-Drinfeld  category and the category of the center of representation category of quasi-Turaev Hopf group coalgebra will be established. Moreover, both of the categories are Turaev braided group categories.

This paper is organized as follows: In section 1, we will recall the notions of crossed $T$-category and its center and quasi-Turaev group coalgebra. In section 2,  we will construct the Yetter-Drinfeld module over the quasi-Turaev group coalgebra and prove that the Yetter-Drinfeld category is isomorphic to the center of the representation category.

Throughout this article, let $k$ be a fixed field. All the algebras and linear spaces are over $k$; unadorned $\o$ means $\o_k$.

\section{Preliminary}
\def\theequation{1.\arabic{equation}}
\setcounter{equation} {0} \hskip\parindent
In this section, we will recall the definitions and notations relevant to Turaev braided group categories.
\subsection{Crossed $T$-category}

A tensor category $\mathcal{C}=(C,\o,a,l,r)$ is a category $\mathcal{C}$ endowed with a functor $\o : \mathcal{C} \times \mathcal{C}\rightarrow \mathcal{C}$ (the tensor product), an object $\mathcal{I} \in \mathcal{C}$(the tensor unit), and natural isomorphisms $a=a_{U,V,W}:(U\o V)\o W\rightarrow U\o(V\o W)$ for all $U,V,W\in\mathcal{C}$(the associativity constraint), $l=l_U:\mathcal{I}\o U\rightarrow U$(the left unit constraint) and $r=r_U:U\o I\rightarrow U$(the right unit constraint) for all $U\in\mathcal{C}$ such that for all $U,V,W,X\in\mathcal{C}$, the associativity pentagon
$$
a_{U,V,W\o X}\circ a_{U\o V,W,X}=(U\o a_{V,W,X})\circ a_{U,V\o W,X}\circ (a_{U,V,W}\o X),
$$
and the triangle
$$(U\o l_V)\circ(r_U\o V)=a_{U,\mathcal{I},V},$$
are satisfied. A tensor category $\mathcal{C}$ is strict when all the constraints are identities.

Let $\pi$ be a group with the unit 1. Recall from \cite{FW} that a crossed category $\mathcal{C}$ (over $\pi$) is given by the following data:

$\bullet $ $\mathcal{C}$ is a tensor category.

$\bullet $ A family of subcategory $\{\mathcal{C_\a}\}_{\a\in\pi}$ such that $\mathcal{C}$ is a disjonit union of this family and that $U\o V\in\mathcal{C}_{\a\b}$ for any $\a,\b\in\pi$, $U\in\mathcal{C}_\a$ and $V\in\mathcal{C}_\b$.

$\bullet $ A group homomorphism $\varphi:\pi\rightarrow aut(\mathcal{C}),\b\mapsto\varphi_\b$, the $conjugation$, where $aut(\mathcal{C})$ is the group of the invertible strict tensor functors from $\mathcal{C}$ to itself, such that $\varphi_\b(\mathcal{C_\a})=\mathcal{C}_{\b\a\b^{-1}}$ for any $\a,\b\in\pi$.

We will use the left index notation in Turaev: Given $\b\in\pi$ and an object $V\in\mathcal{C_\a}$, the functor $\varphi_\b$ will be denoted by $^{\b}(\cdot)$ or $^{V}(\cdot)$ and $^{\b^{-1}}(\cdot)$ will be denoted by $^{\overline{V}}(\cdot)$. Since $^{V}(\cdot)$ is a functor, for any object $U\in\mathcal{C}$ and any composition of morphism $g\circ f$ in $\mathcal{C}$, we obtain $^{V}id_U=id_{^{V}U}$ and $^{V}(g\circ f)=\ ^{V}g\circ\ ^{V}f$. Since the conjugation $\varphi:\pi\rightarrow aut(\mathcal{C})$ is a group homomorphism, for any $V,W\in\mathcal{C}$, we have $^{V\o W}(\cdot)=\ ^{V}(^{W}(\cdot))$ and $^{1}(\cdot)=\ ^{V}(^{\overline{V}}(\cdot))=\ ^{\overline{V}}(^{V}(\cdot))=id_{\mathcal{C}}$. Since for any $V\in\mathcal{C}$, the functor $^{V}(\cdot)$ is strict, we have $^{V}(f\o g)=\ ^{V}f\o\ ^{V}g$ for any morphism $f$ and $g$ in $\mathcal{C}$, and $^{V}(1)=1.$

A $Turaev ~braided~\pi$-$category$ is a crossed $T$-category $\mathcal{C}$ endowed with a braiding, i.e., a family of isomorphisms
$$c=\{c_{U,V}:U\o V\rightarrow~ ^{V}U\o V\}_{U,V\in\mathcal{C}}
$$
obeying the following conditions:

$\bullet $ For any morphism $f\in Hom_{\mathcal{C}_\a}(U,U')$ and $g\in Hom_{\mathcal{C}_\b}(V,V')$, we have
$$(^{\a}g\o f)\circ c_{U,V}=c_{U',V'}\circ(f\o g),
$$

$\bullet $ For all $U,V,W\in\mathcal{C}$, we have
\begin{eqnarray}
c_{U,V\o W}&=a^{-1}_{^{U}V,^{U}W,U}\circ(^{U}V\o c_{U,W})\circ a_{^{U}V,U,W}\circ(c_{U,V}\o W)\circ a^{-1}_{U,V,W},\label{A}\\
c_{U\o V,W}&=a_{^{U\o V}W,U,V}\circ(c_{U,^{V}W}\o V)\circ a^{-1}_{U,^{V}W,V}\circ(U\o c'_{V,W})\circ a_{U,V,W}.
\end{eqnarray}

$\bullet $ For any $U,V\in\mathcal{C}$ and $\a\in\pi$, $\varphi_\a(c_{U,V})=c_{^{\a}U,^{\a}V}$.

\subsection{The center of a crossed $T$-category}
Let $\mathcal{C}$ be a crossed $T$-category. The center of $\mathcal{C}$ is the braided crossed $T$-category $\mathcal{Z(C)}$ defined as follows:
\begin{enumerate}
\item
The objects of $\mathcal{Z(C)}$ are the pairs $(U,c_{U,-})$ satisfying the following conditions:

$\bullet$ $U$ is an object of $\mathcal{C}$.

$\bullet$ $c_{U,-}$ is a natural isomorphism from the functor $U\o -$ to the functor $^{U}(-)\o U$ such that for any objects $V,W\in\mathcal{C}$, the identity (\ref{A}) is satisfied.

\item
The morphism in $\mathcal{Z(C)}$ from $(U,c_{U,-})$ to $(V,c'_{V,-})$ is a morphism $f:U\rightarrow V$ such that for any object $X\in\mathcal{C}$,
\begin{equation}
(^{U}X\o f)\circ c_{U,X}=c'_{V,X}\circ(f\o X).\label{B}
\end{equation}
The composition of two morphisms in $\mathcal{Z(C)}$ is given by the composition in $\mathcal{C}$.

\item Given $Z_1=(U,c_{U,-})$ and $Z_2=(V,c'_{V,-})$ in $\mathcal{Z(C)}$, the tensor product $Z_1\o Z_2$ in $\mathcal{Z(C)}$ is the couple $(U\o V,(c\o c')_{U\o V,-})$, where for any object $W\in\mathcal{C}$, $(c\o c')_{U\o V,-}$ is obtained by
\begin{eqnarray}
(c\o c')_{U\o V,W}=a_{^{U\o V}W,U,V}\circ(c_{U,^{V}W}\o V)\circ a^{-1}_{U,^{V}W,V}\circ(U\o c'_{V,W})\circ a_{U,V,W}.\label{C}
\end{eqnarray}

\item The unit of $\mathcal{Z(C)}$ is the couple $(I,id_{-})$, where $I$ is the unit of $\mathcal{C}$.

\item For any $\a\in\pi$, the $\a$th component of $\mathcal{Z(C)}$, denoted $\mathcal{Z_{\a}(C)}$, is the full subcategory of $\mathcal{Z(C)}$ whose objects are the pairs $(U,c_{U,-})$, where $U\in\mathcal{C_{\a}}$.

\item For any $\b\in\pi$, the automorphism $\varphi_{\mathcal{Z}.\b}$ is given by, for any $(U,c_{U,-})\in\mathcal{Z(C)}$,
\begin{eqnarray}
\varphi_{\mathcal{Z}.\b}(U,c_{U,-})=(\varphi_{\b}(U),\varphi_{\mathcal{Z}.\b}(c_{U,-})),\label{D}
\end{eqnarray}
where $\varphi_{\mathcal{Z}.\b}(c_{U,-})_{\varphi_{\b}(U),X}=\varphi_{\b}(c_{U,\varphi^{-1}_{\b}(X)})$ for any $X\in\mathcal{C}$.

\item The braiding $c$ in $\mathcal{Z(C)}$ is obtained by setting $c_{Z_1,Z_2}=c_{U,V}$ for any $Z_1=(U,c_{U,-}),Z_2=(V,c'_{V,-})\in\mathcal{Z(C)}$.
\end{enumerate}

\subsection{Quasi-Turaev group coalgebras}

Recall from \cite{FW}, a family of algebras $H=\{H_\a\}_{\a\in\pi}$ is a quasi-semi-T-coalgebra if there exist a family of morphisms of algebra $\Delta=\{\Delta_{\a,\b}:H_{\a\b}\rightarrow H_\a\o H_\b\}_{\a,\b\in\pi}$, a morphism of algebra $\varepsilon:H_1\rightarrow k$ and a family of invertible elements $\{\Phi_{\a,\b,\g}\in H_\a\o H_\b\o H_\g\}_{\a,\b,\g\in\pi}$ such that
\begin{eqnarray}
&&(H_\a\o\Delta_{\b,\g})\Delta_{\a,\b\g}(h)\Phi_{\a,\b,\g}=\Phi_{\a,\b,\g}(\Delta_{\a,\b}\o H_\g)\Delta_{\a\b,\g}(h),\\
&&(H_\a\o\varepsilon)(\Delta_{\a,1}(a))=a,\ (\varepsilon\o H_\a)(\Delta_{1,\a}(a))=a,\\
&&(1_\a\o\Phi_{\b,\g,\l})(H_\a\o\Delta_{\b,\g}\o H_\l)(\Phi_{\a,\b\g,\l})(\Phi_{\a,\b,\g}\o1_\l)\nonumber\\
&&~~~~=(H_\a\o H_\b\o\Delta_{\g,\l})(\Phi_{\a,\b,\g\l})(\Delta_{\a,\b}\o H_\g\o H_\l)(\Phi_{\a\b,\g,\l}),\\
&&(H_\a\o\varepsilon\o H_\b)(1_\a\o1_1\o1_\b)=1_\a\o1_\b
\end{eqnarray}
for all $h\in H_{\a\b\g}$ and $a\in H_\a.$ $\Delta$ is called $comultiplication$, and $\varepsilon$ the $counit.$

In our computations, we will use the Sweedler-Heyneman notation $\Delta_{\alpha,\beta}(b)= b _{(1,\alpha)}\o b_{(2,\beta)}$
for all $b\in H_{\alpha\beta}$ (summation implicitely understood). Since $\Delta$ is only quasi-coassociative, we adopt further convention
$$(id_{\alpha}\o \Delta_{\beta,\gamma})\Delta_{\alpha,\beta\gamma}(h)= h_{(1,\alpha)}\o h_{(2,\beta\gamma)(1,\beta)}\o h_{(2,\beta\gamma)(2,\gamma)},$$
$$(\Delta _{\alpha,\beta}\o id_{\gamma})\Delta_{\alpha\beta,\gamma}(h)= h_{(1,\alpha\beta)(1,\alpha)}\o h_{(1,\alpha\beta)(2,\beta)}\o h_{(2,\gamma)},$$
for all $h\in H_{\alpha\beta\gamma}$. We will denote the components of $\Phi$ by capital letters, and the ones of $\Phi^{-1}$ by small letters, namely,
$$\Phi_{\alpha,\beta,\gamma}=  Y^{1}_{\alpha} \o Y^{2}_{\beta}\o Y^{3}_{\gamma}= T^{1}_{\alpha} \o T^{2}_{\beta}\o T^{3}_{\gamma}=\cdots$$
$$\Phi^{-1}_{\alpha,\beta,\gamma}=  y^{1}_{\alpha} \o y^{2}_{\beta}\o y^{3}_{\gamma}= t^{1}_{\alpha} \o t^{2}_{\beta}\o t^{3}_{\gamma}=\cdots$$

A quasi-Hopf group coalgebra is a quasi-semi-T-coalgebra $H=(\{H_\a\}_{\a\in\pi},\Delta,\varepsilon)$ endowed with a family of invertible anti-automorphisms of algebra $S=\{S_\a:H_\a\rightarrow H_{\a^{-1}}\}_{\a\in\pi}$ $(the\ antipode)$ and elements $\{p_\a,q_\a\in H_\a\}_{\a\in\pi}$ such that the following conditions hold:
\begin{eqnarray}
&&S_\a(h_{(1,\a)})p_{\a^{-1}}h_{(2,\a^{-1})}=\varepsilon(h)p_{\a^{-1}},~~~h_{(1,\a)}q_\a S_{\a^{-1}}(h_{(2,\a^{-1})})=\varepsilon(h)q_\a,\\
&&Y^1_\a q_\a S_{\a^{-1}}(Y^2_{\a^{-1}})p_\a Y^3_\a=1_\a,~~~~~~~~~~~S_{\a^{-1}}(y^1_{\a^{-1}})p_\a y^2_\a q_\a S_{\a^{-1}}(y^3_{\a^{-1}})=1_\a.
\end{eqnarray}

A quasi-Turaev $\pi$-coalgebra is a quasi-Hopf $\pi$-coalgebra $H=(\{H_\a\}_{\a\in\pi},\Delta,\varepsilon,\Phi)$ with a family of $k$-linear maps $\varphi=\{\varphi_\b:H_\a\rightarrow H_{\b\a\b^{-1}}\}_{\a,\b\in\pi}$(the crossing) such that the following conditions hold:

$\bullet$ For any $\b\in\pi$, $\varphi_\b$ is an algebra isomorphism.

$\bullet$ $\varphi_\b$ preserves the comultiplication and the counit, i.e., for any $\a,\b,\g\in\pi$,
$$(\varphi_\b\o\varphi_\b)\Delta_{\a,\g}=\Delta_{\b\a\b^{-1},\b\g\b^{-1}}\circ\varphi_\b,
$$
$$\varepsilon\circ\varphi_\b=\varepsilon.
$$

$\bullet$ $\varphi$ is multiplicative in the sense that $\varphi_\b\varphi_{\b'}=\varphi_{\b\b'}$ for all $\b,\b'\in\pi.$

$\bullet$ The family $\Phi$ is invariant under the crossing, i.e., for any $\Phi_{\a,\b,\g}$,
$$(\varphi_\eta\o\varphi_\theta\o\varphi_\vartheta)\Phi_{\a,\b,\g}=\Phi_{\eta\a\eta^{-1},\theta\b\theta^{-1},\vartheta\g\vartheta^{-1}}.
$$

\section{Main results}
\def\theequation{2.\arabic{equation}}
\setcounter{equation} {0} \hskip\parindent

In this section, we will give the main result of this paper. First of all, we need some preparations.
For any Hopf group coalgebra $H=(\{H_\a\},\Delta,\varepsilon,S)$, we obviously have the following identity
$$h_{(1,\a)}\o h_{(2,\b)}S_{\b^{-1}}(h_{(3,\b^{-1})})=h\o1_\b,$$
for all $\a,\b\in\pi$ and $h\in H_\a$. We will need the generalization of this formula to the quasi-Hopf group coalgebra
setting. The following lemma will be given without proof.

\begin{lemma}
Let $H=(\{H_\a\},\Delta,\varepsilon,S)$ be a quasi-Hopf group coalgebra. Set
\begin{eqnarray}
&&I^R_{\a,\b}=I^1_\a\o I^2_\b= y^1_\a\o y^2_\b q_\b S_{\b^{-1}}( y^3_{\b^{-1}}),\\
&&J^R_{\a,\b}=J^1_\a\o J^2_\b= Y^{1}_\a\o S^{-1}_{\b}(p_{\b^{-1}} Y^3_{\b^{-1}}) Y^{2}_\b,\\
&&I^L_{\a,\b}=\tilde{I}^1_\a\o \tilde{I}^2_\b= Y^{2}_\a S^{-1}_\a( Y^{1}_{\a^{-1}}q_{\a^{-1}})\o  Y^{3}_{\b},\\
&&J^L_{\a,\b}=\tilde{J}^1_\a\o \tilde{J}^2_\b=S_{\a^{-1}}( y^1_{\a^{-1}})p_\a y^2_\a \o  y^3_{\b}.
\end{eqnarray}
Then for all $h\in H_{\a}$ and $a\in H_\b$, we have
\begin{eqnarray}
&&\Delta_{\a,\b}(h_{(1,\a\b)})I^R_{\a,\b}[1\o S_{\b^{-1}}(h_{(2,\b^{-1})})]=I^R_{\a,\b}[h\o 1],\\
&&[1\o S^{-1}_{\b}(h_{(2,\b^{-1})})]J^R_{\a,\b}\Delta_{\a,\b}(h_{(1,\a\b)})=[h\o 1]J^R_{\a,\b},\\
&&\Delta_{\a,\b}(a_{(2,\a\b)})I^L_{\a,\b}[S^{-1}_{\a}(a_{(1,\a^{-1})})\o 1]=I^L_{\a,\b}[1\o a],\\
&&[S_{\a^{-1}}(a_{(1,\a^{-1})})\o1]J^L_{\a,\b}\Delta_{\a,\b}(a_{(2,\a\b)})=J^L_{\a,\b}[1\o a].
\end{eqnarray}

And the following relations hold:
\begin{eqnarray}
&&\Delta_{\a,\b}(J^1_{\a\b})I^R_{\a,\b}[1_\a\o S_{\b^{-1}}(J^2_{\b^{-1}})]=1_\a\o1_\b,\\
&&[1_\a\o S^{-1}_{\b}(I^2_{\b^{-1}})]J^R_{\a,\b}\Delta_{\a,\b}(I^1_{\a\b})=1_\a\o1_\b,\\
&&\Delta_{\a,\b}(\tilde{J}^2_{\a\b})I^L_{\a,\b}[S^{-1}_{\a}(\tilde{J}^1_{\a^{-1}})\o 1_\b]=1_\a\o1_\b,\\
&&[S_{\a^{-1}}(\tilde{I}^1_{\a^{-1}})\o1_\b]J^L_{\a,\b}\Delta_{\a,\b}(\tilde{I}^2_{\a\b})=1_\a\o1_\b.
\end{eqnarray}
\end{lemma}

In \cite{Z}, M. Zunino defined the Yetter-Drinfeld module over the crossed group coalgebra, and S. Majid in \cite{M} ingeniously constructed the Yetter-Drinfeld module over quasi-Hopf algebra. With these help, we have the following definition.

\begin{definition}
Fix an element $\alpha\in\pi$. An $\a$-Yetter-Drinfeld module or YD$_{\alpha}$-module is a couple $V=\{V,\rho_V=\{\rho_{V,\l}\}_{\l\in\pi}\}$, where $\rho_{V,\l}:V\rightarrow V\o H_\l,\ v\mapsto v_{(0,0)}\o v_{(1,\l)}$ is a $k$-linear morphism such that the following conditions are satisfied:
\begin{enumerate}
\item
$V$ is a left $H_\a$-module,
\item
$V$ is counitary in the sense that
\begin{equation}
(id\o\varepsilon)\circ\rho_{V,1}=id.
\end{equation}
\item
For all $v\in V$,
\begin{eqnarray}
\lefteqn{( y^2_{\alpha}\cdot v_{(0,0)})_{(0,0)}\o ( y^2_{\alpha}\cdot v_{(0,0)})_{(1,\lambda_1)} y^1_{\lambda_1}\o y^3_{\lambda_2}v_{(1,\lambda_2)}}\label{E}\\
&=\Phi^{-1}_{\a,\l_1,\l_2}\cdot[( y^3_{\alpha}\cdot v)_{(0,0)}\o( y^3_{\alpha}\cdot v)_{(1,\l_1\l_2)(1,\l_1)} y^1_{\lambda_1}
\o( y^3_{\alpha}\cdot v)_{(1,\l_1\l_2)(2,\l_2)} y^2_{\lambda_2}].\nonumber
\end{eqnarray}
\item
For all $h\in H_{\a\b}$ and $v\in V$,
\begin{equation}
h_{(1,\a)}\cdot v_{(0,0)}\o h_{(2,\b)}v_{(1,\b)}=(h_{(2,\a)}\cdot v)_{(0,0)}\o(h_{(2,\a)}\cdot v)_{(1,\b)}\varphi_{\a^{-1}}(h_{(1,\a\b\a^{-1})}).\label{F}
\end{equation}
\end{enumerate}
\end{definition}

\begin{remark}
Note that in the above definition, when the quasi-Hopf group coalgebra is trivial, i.e., $\varphi_{\a,\b,\l}=1_\a\o1_\b\o1_\l$ for any $\a,\b,\l\in\pi$, then we have a YD$_\a$-module over Hopf group coalgebra introduced in \cite{Z}.
\end{remark}

Given two YD$_{\a}$-modules $(U,\rho_U)$ and $(V,\rho_V)$, a linear map $f:U\rightarrow V$ is said to be a morphism of YD$_{\a}$-module if $f$ is $H_{\a}$-linear and for any $\l\in\pi$,
$$
\rho_{V,\l}\circ f=(f\o H_\l)\circ\rho_{U,\l}.
$$

Let YD$(H)$ be the disjoint union of the categories YD$_\a(H)$ for all $\a\in\pi$. The category
YD$(H)$ admits the structure of a braided $T$-category as follows:

$\bullet$ The tensor product of a YD$_\a$-module $(V,\rho_V)$ and a YD$_\b$-module $(W,\rho_W)$ is a YD$_{\a\b}$-module $(V\o W,\rho_{V\o W})$, where for any $v\in V,w\in W$ and $\l\in\pi$,
\begin{eqnarray}
\rho_{V\o W}(v\o w)&=&t^1_\a Y^1_\a\cdot(y^2_\a\cdot v)_{(0,0)}\o t^2_\b\cdot(Y^3_\b y^3_\b\cdot w)_{(0,0)}\nonumber\\
                   &&\o t^3_\l(Y^3_\b y^3_\b\cdot w)_{(1,\l)}Y^2_\l\varphi_{\b^{-1}}((y^2_\a\cdot v)_{(1,\b\l\b^{-1})})y^1_\l.\label{G}
\end{eqnarray}
The unit of YD$(H)$ is the pair $(k,\rho_k)$, where for any $\l\in\pi$, $\rho_\l(1)=1\o 1_\l$. Then the tensor product of arrows is given by the tensor product of $k$-linear maps.

$\bullet$ For any $\b\in\pi$, the conjugation functor $^{\b}(\cdot)$ is given as follows. Let $(V,\rho_V)$ be a YD$_\a$-module and we set $^{\b}(V,\rho_V)=(^{\b}V,\rho_{^{\b}V})$, where for any $\l\in\pi$ and $v\in V$,
\begin{equation}
\rho_{^{\b}V}(v)=\ ^{\b}((^{\b^{-1}}v)_{(0,0)})\o \varphi_{\b}((^{\b^{-1}}v)_{(1,\b^{-1}\l\b)}).\label{H}
\end{equation}
For any morphism $f:(V,\rho_V)\rightarrow(W,\rho_W)$ of YD-module and any $v\in V$, we set $(^{\b}f)(^{\b}v)=\ ^{\b}(f(v))$.

$\bullet$ For any YD$_\a$-module $(V,\rho_V)$ and any YD$_\b$-module $(W,\rho_W)$, the braiding $c$ is given by
\begin{equation}
c_{V,W}(v\o w)=\ ^{\a}[J^1_{(1,\b)} y^1_{\b}S_{\b^{-1}}(J^2_{\b^{-1}} y^3_{\b^{-1}}(\tilde{I}^2_\a\cdot v)_{(1,\b^{-1})}\tilde{I}^1_{\b^{-1}})\cdot w]\o J^1_{(2,\a)} y^2_\a\cdot(\tilde{I}^2_\a\cdot v)_{(0,0)}.\label{I}
\end{equation}

\begin{lemma}
For a fixed element $\a\in\pi$, let $(V,c_{V,-})$ be any object in $\mathcal{Z}_{\a}(Rep(H))$. For any $\l\in\pi$, define the linear map $\rho_{V,\l}:V\rightarrow V\o H_\l$ by
\begin{equation}
\rho_{V,\l}(v)=c^{-1}_{V,H_{\l}}(^{\a}1_\l\o v).
\end{equation}
Then the pair $V=(V,\rho_V=\{\rho_{V,\l}\}_{\l\in\pi})$ is a YD$_\a$-module. Hence we have a functor $F_1:\mathcal{Z}(Rep(H))\rightarrow$ YD$(H)$ given by $F_1(V,c_{V,-})=(V,\rho_V)$ and $F_1(f)=f$, where $f$ is a morphism in $\mathcal{Z}(Rep(H))$.
\end{lemma}

\begin{proof}
We just need to verify that $(V,\rho_V)$ satisfies the axioms of YD$_\a$-modules.

First of all, for any $\l_1,\l_2\in\pi$, consider $H_{\l_1}$ and $H_{\l_2}$ as the modules over themselves. By (\ref{A}), we have
$$
a^{-1}_{V,H_{\l_1},H_{\l_2}}\circ c^{-1}_{V,H_{\l_1}\o H_{\l_2}}\circ a^{-1}_{^{V}H_{\l_1},^{V}H_{\l_2},V}=(c^{-1}_{V,H_{\l_1}}\o H_{\l_2})\circ a^{-1}_{^{V}H_{\l_1},V,H_{\l_2}}\circ(^{V}H_{\l_1}\o c^{-1}_{V,H_{\l_2}}).
$$
For all $v\in V$, both of the sides evaluating at $^{\a}1_{\l_1}\o\ ^{\a}1_{\l_2}\o v$, we have
\begin{eqnarray*}
\lefteqn{(y^2_{\alpha}\cdot v_{(0,0)})_{(0,0)}\o (y^2_{\alpha}\cdot v_{(0,0)})_{(1,\lambda_1)}y^1_{\lambda_1}\o y^3_{\lambda_2}v_{(1,\lambda_2)}}\\
&=y_{\a,\l_1,\l_2}\cdot[(y^3_{\alpha}\cdot v)_{(0,0)}\o(y^3_{\alpha}\cdot v)_{(1,\l_1\l_2)(1,\l_1)}y^1_{\lambda_1}
\o(y^3_{\alpha}\cdot v)_{(1,\l_1\l_2)(2,\l_2)}y^2_{\lambda_2}].
\end{eqnarray*}
The counitarity of $V$ is obvious.

Secondly for all $v\in V$ and $h\in H_{\a\l}$, we have on one hand,
$$
h\cdot c^{-1}_{V,H_\l}(^{\a}1_\l\o v)=h\cdot(v_{(0,0)}\o v_{(1,\l)})=h_{(1,\a)}\cdot v_{(0,0)}\o h_{(2,\l)}v_{(1,\l)},
$$
and on the other hand,
\begin{eqnarray*}
c^{-1}_{V,H_\l}(h\cdot(^{\a}1_\l\o v))&=&c^{-1}_{V,H_\l}(h_{(1,\a\l\a^{-1})}\cdot\ ^{\a}1_\l\o h_{(2,\a)}\cdot v))\\
                                      &=&c^{-1}_{V,H_\l}(^{\a}(\varphi_{\a^{-1}}(h_{(1,\a\l\a^{-1}})))\o h_{(2,\a)}\cdot v))\\
                                      &=&(h_{(2,\a)}\cdot v)_{(0,0)}\o(h_{(2,\a)}\cdot v)_{(1,\l)}\varphi_{\a^{-1}}(h_{(1,\a\l\a^{-1})}).
\end{eqnarray*}
Since the braiding $c_{V,H_\l}$ is $H$-linear, we obtain
$$
h_{(1,\a)}\cdot v_{(0,0)}\o h_{(2,\l)}v_{(1,\l)}=(h_{(2,\a)}\cdot v)_{(0,0)}\o(h_{(2,\a)}\cdot v)_{(1,\l)}\varphi_{\a^{-1}}(h_{(1,\a\l\a^{-1})}).
$$
Finally, let $f:(V,c_{V,-})\rightarrow (W,c_{W,-})$ is a morphism in $\mathcal{Z}_\a(Rep(H))$, then as the case of Hopf group coalgebra, $f$ gives rise to a morphism of YD$_\a$-module. It is easy to see  that $F_1$ is a functor. This completes the proof.
\end{proof}

Assume that $(V,\rho_V)$ is an object in the category YD$_\a(H)$. For any $\l\in\pi$ and left $H_\l$-module $X$, give the linear map $c_{V,X}:V\o X\rightarrow\ ^{\a}X\o V$ by
$$
c_{V,X}(v\o x)=\ ^{\a}[J^1_{(1,\l)}y^{1}_{\l}S_{\l^{-1}}(J^2_{\l^{-1}}y^{3}_{\l^{-1}}(\tilde{I}^2_\a\cdot v)_{(1,\l^{-1})}\tilde{I}^1_{\l^{-1}})\cdot x]\o J^1_{(2,\a)}y^{2}_\a\cdot(\tilde{I}^2_\a\cdot v)_{(0,0)},
$$
for all $v\in V$ and $x\in X$.
\begin{lemma}
The couple $(V,c_{V,-})$ is an object in $\mathcal{Z}(Rep(H))$. Hence we have a functor $F_2:$YD$(H)\rightarrow\mathcal{Z}(Rep(H))$  given by $F_2(V,\rho_V)=(V,c_{V,-})$ and $F_2(f)=f$, where $f$ is a morphism in YD$(H)$. The functors $F_1$ and $F_2$ are inverses.
\end{lemma}

\begin{proof}
Firstly for any $\l\in\pi$ and left $H_\l$-module $X$, we set a morphism $\hat{c}_{V,X}:\ ^{\a}X\o V\rightarrow V\o X$ by
$$
\hat{c}_{V,X}(^{\a}x\o v)=v_{(0,0)}\o v_{(1,\l)}\cdot x.
$$
Then
$$\begin{aligned}
&\hat{c}_{V,X}\circ c_{V,X}(v\o x)\\
&=\hat{c}_{V,X}(\ ^{\a}[J^1_{(1,\l)}y^1_{\l}S_{\l^{-1}}(J^2_{\l^{-1}}y^3_{\l^{-1}}(\tilde{I}^2_\a\cdot v)_{(1,\l^{-1})}\tilde{I}^1_{\l^{-1}})\cdot x]\o J^1_{(2,\a)}y^2_\a\cdot(\tilde{I}^2_\a\cdot v)_{(0,0)})\\
&=\underline{[J^1_{(2,\a)}y^2_\a\cdot(\tilde{I}^2_\a\cdot v)_{(0,0)})]_{(0,0)}}\o\\
&\ \ \underline{[J^1_{(2,\a)}y^2_\a\cdot(\tilde{I}^2_\a\cdot v)_{(0,0)})]_{(1,\l)}\varphi_{\a^{-1}}(J^1_{(1,\a\l\a^{-1})})}\varphi_{\a^{-1}}(y^1_{\a\l\a^{-1}})\cdot[S_{\l^{-1}}(J^2_{\l^{-1}}y^3_{\l^{-1}}(\tilde{I}^2_\a\cdot v)_{(1,\l^{-1})}\tilde{I}^1_{\l^{-1}})\cdot x]\\
&\stackrel {(\ref{E})}{=}J^1_{(1,\a)}\cdot\underline{(y^2_\a\cdot(\tilde{I}^2_\a\cdot v)_{(0,0)})_{(0,0)}}\o\\
&\ \ J^1_{(2,\l)}\underline{(y^2_\a\cdot(\tilde{I}^2_\a\cdot v)_{(0,0)})_{(1,\l)}y^1_{\l}}S_{\l^{-1}}(J^2_{\l^{-1}}\underline{y^3_{\l^{-1}}(\tilde{I}^2_\a\cdot v)_{(1,\l^{-1})}}\tilde{I}^1_{\l^{-1}})\cdot x\\
&=J^1_{(1,\a)}t^{1}_\a\cdot(y^3_{\a}\tilde{I}^2_\a\cdot v)_{(0,0)}\o\\
&\ \ J^1_{(2,\l)}t^2_\l(y^3_{\a}\tilde{I}^2_\a\cdot v)_{(1,1)(1,\l)}y^1_{\l}S_{\l^{-1}}(J^2_{\l^{-1}}t^{-3}_{\l^{-1}}(y^3_{\a}\tilde{I}^2_\a\cdot v)_{(1,1)(2,\l^{-1})}y^2_{\l^{-1}})\tilde{I}^1_{\l^{-1}})\cdot x\\
&=J^1_{(1,\a)}t^{1}_\a\cdot(y^3_{\a}\tilde{I}^2_\a\cdot v)_{(0,0)}\o\\
&\ \ J^1_{(2,\l)}t^2_\l(y^3_{\a}\tilde{I}^2_\a\cdot v)_{(1,1)(1,\l)}y^1_{\l}S_{\l^{-1}}(y^2_{\l^{-1}}\tilde{I}^1_{\l^{-1}})S_{\l^{-1}}(J^2_{\l^{-1}}t^{-3}_{\l^{-1}}(y^3_{\a}\tilde{I}^2_\a\cdot v)_{(1,1)(2,\l^{-1})})\cdot x\\
&=J^1_{(1,\a)}t^{1}_\a\cdot v_{(0,0)}\o
J^1_{(2,\l)}t^2_\l v_{(1,1)(1,\l)}q_\l S_{\l^{-1}}(v_{(1,1)(2,\l^{-1})}) S_{\l^{-1}}(J^2_{\l^{-1}}t^{-3}_{\l^{-1}} )\cdot x\\
&=J^1_{(1,\a)}t^{1}_\a\cdot v\o
J^1_{(2,\l)}t^2_\l q_\l S_{\l^{-1}}(J^2_{\l^{-1}}t^{-3}_{\l^{-1}} )\cdot x\\
&= v\o x.
\end{aligned}$$
Hence $\hat{c}_{V,X}\circ c_{V,X}=id_{V\o X}$.  Similarly $\hat{c}_{V,X}\circ c_{V,X}=id_{^{\a}X\o V}$. Therefore $\hat{c}_{V,X}$ and $c_{V,X}$ are inverses.

Secondly for any $h\in H_{\a\l}$,
\begin{eqnarray*}
h\cdot c^{-1}_{V,X}(^{\a}x\o v)&=&h_{(1,\a)}\cdot v_{(0,0)}\o h_{(2,\l)}v_{(1,\l)}\cdot x \\
                              &\stackrel {(\ref{F})}{=}&(h_{(2,\a)}\cdot v)_{(0,0)}\o(h_{(2,\a)}\cdot v)_{(1,\b)}\varphi_{\a^{-1}}(h_{(1,\a\l\a^{-1})})\cdot x\\
                              &=&c^{-1}_{V,X}(^{\a}(\varphi_{\a^{-1}}(h_{(1,\a\l\a^{-1})})\cdot x)\o h_{(2,\a)}\cdot v)\\
                              &=&c^{-1}_{V,X}(h_{(1,\a\l\a^{-1})}\cdot\ ^{\a}x\o h_{(2,\a)}\cdot v)\\
                              &=&c^{-1}_{V,X}(h\cdot(^{\a}x\o v)).
\end{eqnarray*}
That is, $c^{-1}_{V,X}$ is $H_{\a\l}$-linear, so is $c_{V,X}$. The naturality of $c_{V,X}$ is straightforward to verify.

Next suppose that $X_1$ is an $H_{\l_1}$-module and $X_2$ an $H_{\l_2}$-module for all $\l_1,\l_2\in\pi$, and for any $x_1\in X_1,x_2\in X_2$,
$$\begin{aligned}
&a_{V,X_1,X_2}\circ(c^{-1}_{V,X_1}\o X_2)\circ a^{-1}_{^{\a}X_1,V,X_2}\circ(^{\a}X_1\o c^{-1}_{V,X_2})\circ a_{^{\a}X_1,^{\a}X_2,V}(^{\a}x_1\o^{\a}x_2\o v)\\
&=T^1_{\a}\cdot[y^2_\a\cdot(Y^{3}_\a\cdot v)_{(0,0)}]_{(0,0)}\o T^2_{\l_1}\cdot[y^2_\a\cdot(Y^{3}_\a\cdot v)_{(0,0)}]_{(1,\l_1)}Y^{-1}_{\l_1}Y^1_{\l_1}\cdot x_1\\
&\quad\o T^3_{\l_2}y^3_{\l_2}(Y^{3}_\a\cdot v)_{(1,\l_2)}Y^{2}_{\l_2}\cdot x_2\\
&\stackrel {(\ref{E})}{=}(y^3_\a Y^{3}_\a\cdot v)_{(0,0)}\o(y^3_\a Y^{3}_\a\cdot v)_{(1,\l_1\l_2)(1,\l_1)}y^1_{\l_1}Y^1_{\l_1}\cdot x_1\o(y^3_\a Y^{3}_\a\cdot v)_{(1,\l_1\l_2)(2,\l_2)}y^2_{\l_2}Y^2_{\l_2}\cdot x_2\\
&=v_{(0,0)}\o v_{(1,\l_1\l_2)(1,\l_1)}\cdot x_1\o v_{(1,\l_1\l_2)(2,\l_2)}y^2_{\l_2}\cdot x_2\\
&=c^{-1}_{V,X_1\o X_2}(^{\a}x_1\o\ ^{\a}x_2\o v).
\end{aligned}$$

Let $V,W$ be YD$_\a$-modules, $f:V\rightarrow W$ be any morphism of YD$_\a$-module. For any $H_\l$-module $X$ and $x\in X$,
$$\begin{aligned}
&c_{W,X}\circ(f\o id)(v\o x)=c_{W,X}(f(v)\o x)\\
                           &=\ ^{\a}[J^1_{(1,\l)}y^{1}_{\l}S_{\l^{-1}}(J^2_{\l^{-1}} y^3_{\l^{-1}}(\tilde{I}^2_\a\cdot f(v))_{(1,\l^{-1})}\tilde{I}^1_{\l^{-1}})\cdot x]\o J^1_{(2,\a)} y^2_\a\cdot(\tilde{I}^2_\a\cdot f(v))_{(0,0)}\\
                           &=\ ^{\a}[J^1_{(1,\l)}y^{1}_{\l}S_{\l^{-1}}(J^2_{\l^{-1}} y^3_{\l^{-1}}( f(\tilde{I}^2_\a\cdot v))_{(1,\l^{-1})}\tilde{I}^1_{\l^{-1}})\cdot x]\o J^1_{(2,\a)} y^2_\a\cdot( f(\tilde{I}^2_\a\cdot v))_{(0,0)}\\
                           &=\ ^{\a}[J^1_{(1,\l)}y^{1}_{\l}S_{\l^{-1}}(J^2_{\l^{-1}} y^3_{\l^{-1}}( \tilde{I}^2_\a\cdot v)_{(1,\l^{-1})}\tilde{I}^1_{\l^{-1}})\cdot x]\o J^1_{(2,\a)} y^2_\a\cdot f((\tilde{I}^2_\a\cdot v)_{(0,0)})\\
                           &=\ ^{\a}(id\o f)c_{V,X}(v\o x).
\end{aligned}$$
That is, $f$ is a morphism in $\mathcal{Z}$(Rep$(H)$). Finally by similar arguments in \cite{Z}, we know that $F_1$ and $F_2$ are inverses. This completes the proof.
\end{proof}

\begin{theorem}
The category YD$(H)$ is isomorphic to the category $\mathcal{Z}$(Rep$(H)$). This isomorphism induces the structure of braided $T$-category on YD$(H)$.
\end{theorem}

\begin{proof}
This isomorphism holds via the functors $F_1$ and $F_2$.

Let $(V,\rho_V)$ be a YD$_\a$-module and $(W,\rho_W)$ be a YD$_\b$-module. Suppose that $(V,c_{V,-})=F_2(V,\rho_V)$ and $(W,c'_{W,-})=F_2(W,\rho_W)$ and set
$$
(V,\rho_V)\o(W,\rho_W)=F_1(F_2(V,\rho_V)\o F_2(W,\rho_W))=F_1(V\o W,(c\o c')_{V\o W,-}).
$$
For any $v\in V,w\in W$, we have
$$\begin{aligned}
&\rho_{V\o W,\l}(v\o w)=((c\o c')_{V\o W,H_\l})^{-1}(^{\a\b}1_\l\o v\o w)\\
&=a^{-1}_{V,W,H_\l}\circ(V\o c'^{-1}_{W,H_\l})\circ a_{V,^{\b}H_\l,W}\circ(c^{-1}_{V,^{\b}H_\l}\o W)\circ a^{-1}_{^{\a\b}H_\l,V,W}(^{\a\b}1_\l\o v\o w)\\
&=y^1_\a Y^1_\a\cdot(t^2_{\a}\cdot v)_{(0,0)}\o y^2_\b\cdot( Y^3_\b t^{3}_\b\cdot w)_{(0,0)}\\
&~~~~~~~~\o y^3_\l( Y^3_\b t^{3}_\b\cdot w)_{(1,\l)} Y^2_\l\varphi_{\b^{-1}}((t^2_{\a}\cdot v)_{(1,\b\l\b^{-1})})t^1_\l,
\end{aligned}$$

where
$$\begin{aligned}
(c^{-1}_{V,^{\b}H_\l}\o W)&(^{\a\b} t^1_\l\o t^2_\a\cdot v\o t^3_\b\cdot w)\\
&=( t^2_\a\cdot v)_{(0,0)}\o\ ^{\b}[\varphi_{\b^{-1}}( t^2_\a\cdot v)_{(1,\b\l\b^{-1})} t^1_\l]\o t^3_\b\cdot w.
\end{aligned}$$

The part concerning the tensor unit of YD$(H)$ is trivial. By similar arguments in \cite{Z}, we can verify the condition (\ref{G})-(\ref{I}). This completes the proof.
\end{proof}

\section*{Acknowledgement}

This work was supported by the NSF of China (No. 11371088) and the NSF of Jiangsu Province (No. BK2012736).

\end{document}